\documentclass[10pt]{article}

\usepackage[english]{babel}
\usepackage[tbtags]{amsmath}
\usepackage{amsfonts,amssymb,amsthm,cite}

\usepackage{color}

\usepackage[width=120mm, left=15mm, top=10mm, height=180mm, paper=a5paper]{geometry}

\usepackage{sectsty}

\sectionfont{\fontsize{10}{13}\selectfont}

\newtheorem{theo}{Theorem}[section]
\newtheorem{lem}{Lemma}[section]
\newtheorem{cor}{Corollary}[section]

\def\Ce#1#2{Cs_{p,#1}(#2)}
\def\Ces{Cs_{p,w}(I)}

\def\Csp{Ch_{p,w}(I)}

\def\loc{\text{\rm loc}}
\def\Ze{\mathbb{Z}}
\def\leb{{\cal L}^1}
\def\supp{\mathop{\rm supp}}
\def\sgn{\mathop{\rm sgn}}

\begin{document}

\begin{center}
 \bf\large On the dual spaces for weighted altered Ces\`aro and Copson  spaces
\end{center}
\begin{center}
Dmitrii V. Prokhorov\\
Computing Center of the Far Eastern Branch\\
of the Russian Academy of Sciences\\
680000 Russia, Khabarovsk, ul. Kim Yu Chena 65\\
prohorov@as.khb.ru
\end{center}

\noindent{\bf Abstract:} We study weighted altered Ces\`aro and Copson spaces, which is non-ideal enlargement of the usual  spaces. We  give full characterization of dual spaces for the spaces.
\vskip 2mm

\vskip 2mm


\section{Introduction}

Let $I:=(c,d)\subset\mathbb{R}$, $p\in(1,\infty)$, $p':=\frac{p}{p-1}$,  ${\cal L}^1$ be Lebesgue measure on $I$ and $\mathfrak{M}(I)$ be the space of all ${\cal L}^1$-measurable functions $f:I\to [-\infty,\infty]$. For $x\in I$ we put $I_{1,x}:=I_{2,x}^*:=(c,x]$ and $I_{1,x}^*:=I_{2,x}:=[x,d)$, and for $\nu\in\{1,2\}$ denote
$$L^1_{\loc,(\nu,1)}:=\Big\{f\in\mathfrak{M}(I):\int_{I_{\nu,x}} |f|<\infty,\,\forall x\in I\Big\},$$
$$L^1_{\loc,(\nu,2)}:=\Big\{f\in\mathfrak{M}(I):\int_{I_{\nu,x}^*} |f|<\infty,\,\forall x\in I\Big\}.$$

Let $w\in \mathfrak{M}(I)$, $w>0$ $\leb$-a.e. on $I$,  
and the space $$\Ces:=\{f\in \mathfrak{M}(I):\|f\|_{\Ces}<\infty\},$$
defined by the seminorm
$$\|f\|_{\Ces}:=\left(\int_I \left[w(x)\int_{I_{\nu,x}}|f|\right]^p\,dx\right)^\frac{1}{p}.$$
The space $\Ces$ is order ideal and has absolutely continuous norm (see \cite[Chapter 1, Sections 3 and 4]{BenSharp-1988}).
In case $\nu=1$ the space $\Ces$ is called the  Ces\`aro space, in case $\nu=2$ the space $\Ces$ is called the Copson space.

The classical Ces\`aro space  $\Ce{\frac{1}{x}}{I}$ (in case of $\nu=1$, $I=(0,\infty)$, $w(x)=\frac{1}{x}$, $x\in I$) was actively studied in works \cite{AM2011,AM,AMint}, to weight spaces $\Ces$ (for $\nu=1$, $I=(0,l)$,  and  weight $w$ with properties:  $w>0$ ${\cal L}^1$-a.e on $I$, $w^p\in L^1_{\loc,(\nu,2)}$ and $\int_I w^p=\infty$) are devoted  the article \cite{KamKub}, some abstract Ces\`aro spaces are considered in \cite{LM}. One of the first problems was the problem of describing the associated space  (see \cite[Chapter 1, Section 2]{BenSharp-1988})  and  dual space for the Ces\`aro space. In \cite[Theorem 3.9]{KamKub} the description of associated space with $\Ces$ is given with help of essential $\int_{I_x^*} w^p$-concave majorant  (see \cite[Definition 2.11]{KamKub}).

Let $(X,\|\cdot\|_X)$ be  normed space of ${\cal L}^1$-measurable functions on $I$. For any $Y\subset X$ and $g\in \mathfrak{M}(I)$  we denote
$$J_{X}(g;Y):=\sup_{h\in Y,\,h\not=0}\frac{|\int_I hg|}{\|h\|_{X}},\ \ \mathbf{J}_{X}(g;Y):=\sup_{h\in Y,\,h\not=0}\frac{\int_I |hg|}{\|h\|_{X}}.$$
If $Y=X$, then we write $J_X(g)$ and $\mathbf{J}_{X}(g)$ instead of $J_{X}(g;X)$ and $\mathbf{J}_{X}(g;X)$, respectively. 
The space associated (K\"othe dual space) with a  normed space $X$ is
$$X'_s:=\{g\in \mathfrak{M}(I):\|g\|_{X'_s}:=\mathbf{J}_X(g)<\infty\}.$$
Note for some spaces the value $J_X(g)$ and $\mathbf{J}_X(g)$ may be differ. Therefore, it makes sense to consider the problem of describing the space 
$$ X'_w:=\{g\in \mathfrak{M}(I):\|g\|_{X'_w}:=J_X(g)<\infty\},$$
which is the subspace $X^*$ of all continuous functionals of the form $f\mapsto \int_I fg$, $f\in X$. The down-spaces $X^\downarrow_{\alpha,s}$, $X^\downarrow_{\alpha,w}$, where $\alpha\in\{s,w\}$, are defined by formulas $Y_\alpha:=\{h\in X'_\alpha: h\ge 0,\  h\downarrow\}$ and
$$X^\downarrow_{\alpha,s}:=\{g\in \mathfrak{M}(I):\|g\|_{X^\downarrow_{\alpha,s}}:=\mathbf{J}_{X'_\alpha}(g;Y_\alpha)<\infty\},$$
$$X^\downarrow_{\alpha,w}:=\{g\in \mathfrak{M}(I):\|g\|_{X^\downarrow_{\alpha,w}}:=J_{X'_\alpha}(g;Y_\alpha)<\infty\}.$$

Remark that $J_{\Ces}(g)=\mathbf{J}_{\Ces}(g)$ for any $g\in \mathfrak{M}(I)$, and hence $(\Ces)'_s=(\Ces)'_w$.

If $w^p\in L^1_{\loc,(\nu,2)}$ and $w>0$ $\leb$-a.e. on $I$ then the set of simple function with compact support in $I$ is subset of $\Ces$,  and there is the embedding $C_0^\infty(I)\subset \Ces$. In this case, since $\Ces$ is order ideal and it has absolutely continuous norm,  for a $\Lambda\in (\Ces)^*$ there exists $g\in (\Ces)'_s$ such that $\|\Lambda\|_{(\Ces)^*}=\|g\|_{(\Ces)'_s}$ and $\Lambda f=\int_I fg$, $f\in\Ces$ (see \cite[Chapter 1, Theorem 4.1]{BenSharp-1988}).

We assume 
\begin{equation}\label{polozh}
I:=(c,d)\subset\mathbb{R},\ p\in (1,\infty),\ \nu\in\{1,2\},\ w^p\in L^1_{\loc,(\nu,2)},\  w>0\ \leb\text{-a.e. on }I.
\end{equation}
In the present paper we study  weighted  altered Ces\`aro\,-\,Copson space 
$$\Csp:=\{f\in L^1_{\loc,(\nu,1)}:  \|f\|_{\Csp}<\infty\},$$
defined by the seminorm 
$$\|f\|_{\Csp}:=\left(\int_I \left|w(x)\int_{I_{\nu,x}}f\right|^p\,dx\right)^\frac{1}{p}.$$

Altered Ces\`aro\,-\,Copson space (sometimes called Ces\`aro\,-\,Copson spaces of non-absolute type) with power weight appears (see \cite{Pr2022,Step2022}) during describing the associated space for the  one weighted Sobolev space of first order on real line \cite{PSU0,PSU-UMN}.

The paper is organized as follows: In Section~2 we prove density $\Ces$ in $\Csp$ and the equality $(\Csp)'_s=\{0\}$. In Section~3 we give full characterization of spaces $(\Csp)'_w$ and  $(\Csp)^*$.

Throughout this paper  
$\mathbb{N}$ stands for the set of positive integers, $\mathbb{Z}$ is the set of all integers, $Df$ is the derivative of function $f$.

\section{Base properties}

\begin{lem}\label{sequence} Let \eqref{polozh} hold.  For any $f\in \Csp$ there exist $\{\alpha_k\}_1^\infty$, $\{\beta_k\}_1^\infty$ such that $\alpha_k,\beta_k\in I$, $\alpha_k\downarrow c$, $\beta_k\uparrow d$ as $k\to\infty$ and
 $$\lim_{k\to\infty}\left|\int_{I_{\nu,\alpha_k}}f\right|^p\int_{I_{\nu,\alpha_k}^*} w^p=\lim_{k\to\infty}\left|\int_{I_{\nu,\beta_k}}f\right|^p\int_{I_{\nu,\beta_k}^*} w^p=0.$$
\end{lem}
 
\begin{proof}
Denote $N:=\sup\{k\in\Ze: 2^k\le\int_I w^p\}+1$ and define $\{a_k\}_{k<N}$ by equality $\int_{I_{\nu,a_k}^*} w^p=2^k$.
Assume that $\nu=2$.

1. If $\int_I w^p<\infty$ then since $f\in L^1_{\loc,(\nu,1)}$ we have $\lim_{s\to d-0}\left|\int_s^df\right|^p\int_c^s w^p=0$
and the required sequence $\{\beta_k\}_1^\infty$ exists.

Let $\int_I w^p=\infty$. Then $N=\infty$ and for $k\ge 1$ we have
\begin{align*}
&\inf_{s\in (a_k,a_{k+1})} \left|\int_s^d f\right|^p\int_c^s w^p\le 2\left[\inf_{s\in (a_k,a_{k+1})} \left|\int_s^df\right|^p\right]\int_{a_k}^{a_{k+1}} w^p\\
&\le 2 \int_{a_k}^{a_{k+1}} \left|w(x)\int_x^d f\right|^p\,dx\le 2 \int_{a_k}^d \left|w(x)\int_x^d f\right|^p\,dx.
\end{align*}
Hence
$$\lim_{k\to\infty}\inf_{s\in (a_k,a_{k+1})} \left|\int_s^df\right|^p\int_c^s w^p=0,$$
and there exists $\beta_k\in (a_k,a_{k+1})$ such that 
$$\left|\int_{\beta_k}^df\right|^p\int_c^{\beta_k} w^p<\frac{1}{k}+\inf_{s\in (a_k,a_{k+1})} \left|\int_s^d f\right|^p\int_c^s w^p.$$
Thus,  $\beta_k\uparrow d$ as $k\to\infty$ and $\lim_{k\to\infty}\left|\int_{\beta_k}^d f\right|^p\int_c^{\beta_k} w^p=0$.

2. For $k\ge 1$ we have
\begin{align*}
\inf_{s\in (a_{-k},a_{-k+1})} \left|\int_s^df\right|^p\int_c^s w^p&\le 2 \int_{a_{-k}}^{a_{-k+1}} \left|w(x)\int_x^d f\right|^p\,dx\\
&\le 2 \int_c^{a_{-k+1}} \left|w(x)\int_x^d f\right|^p\,dx.
\end{align*}
And analogously step 1 we find a sequence $\{\alpha_k\}_1^\infty$ such that  $\alpha_k\downarrow c$ as $k\to\infty$ and $\lim_{k\to\infty}\left|\int_{\alpha_k}^df\right|^p\int_c^{\alpha_k} w^p=0$.

The case $\nu=1$ is proved, analogously.
\end{proof}

\begin{cor}\label{dences}  Let \eqref{polozh} hold. Then $\Ces$ is dense in $\Csp$.
\end{cor}

\begin{proof} Assume that $\nu=2$. 
  Fix $f\in \Csp$. Let   $\{\alpha_k\}_1^\infty$, $\{\beta_k\}_1^\infty$ be the sequences existence of which is proved in Lemma \ref{sequence}. For  $n\in\mathbb{N}$ we put $f_n:=f\chi_{[\alpha_n,\,\beta_n]}$. Then $f_n\in \Ces$ and 
\begin{align*}
 &\|f-f_n\|_{\Csp}^p=\left[\int_{\beta_n}^d+\int_{\alpha_n}^{\beta_n}+\int_c^{\alpha_n}\right]\,\left|w(x)\int_{I_{\nu,x}}(f-f_n)\right|^p\,dx\\
 &=\int_{\beta_n}^d\left|w(x)\int_{I_{\nu,x}}f\right|^p\,dx+ \left|\int_{I_{\nu,\beta_n}}f\right|^p\int_{\alpha_n}^{\beta_n}w^p+\\
 &+\int_c^{\alpha_n} w(x)^p\left|\int_{I_{\nu,x}}f-\int_{I_{\nu,\alpha_n}}f+\int_{I_{\nu,\beta_n}}f\right|^p\,dx\\
 &\le \int_{\beta_n}^d\left|w(x)\int_{I_{\nu,x}}f\right|^p\,dx+4^p\left|\int_{I_{\nu,\beta_n}}f\right|^p\int_{I_{\nu,\beta_n}^*}w^p+\\
 &+4^p\left|\int_{I_{\nu,\alpha_n}}f\right|^p\int_{I_{\nu,\alpha_n}^*}w^p+4^p\int_c^{\alpha_n}\left|w(x)\int_{I_{\nu,x}}f\right|^p\,dx,
 \end{align*}
 that is $\lim_{n\to\infty}\|f-f_n\|_{\Csp}=0$.
 The case $\nu=1$ is proved, analogously.
\end{proof}

\begin{lem}\label{chasto}
 Let \eqref{polozh} hold, $[a,b]\subset I$ and $h\in L^1([a,b])$. For any $\varepsilon >0$ there exists $f\in \Ces$ such that $|f|=|h|$ on $(a,b)$, $\supp f\subset[a,b]$ and $\|f\|_{\Csp}<\varepsilon$. 
\end{lem}

\begin{proof} Fix an arbitrary $\varepsilon>0$. Choose  $n>\frac{1}{\varepsilon}\int_a^b|h|\left[\int_a^b w^p\right]^\frac{1}{p}$. Let $\{a_i\}_{i=0}^{n}$ be partition of $[a,b]$ such that $\int_{a_i}^{a_{i+1}}|h|=\frac{1}{n}\int_a^b|h|$. For each $i\in\{0,\ldots,n-1\}$ we choose
 $b_i\in[a_i,a_{i+1}]$ with property $\int_{a_i}^{b_i}|h|=\int_{b_i}^{a_{i+1}}|h|$.
 Put
 $f:=|h|\sum_{i=0}^{n-1}(\chi_{[a_i,b_i]}-\chi_{(b_i,a_{i+1})})$. Then $f\in \Ces$, $|f|=|h|$ on $(a,b)$ and 
 \begin{align*}
 &\|f\|_{\Csp}^p=\int_a^bw(x)^p\left|\int_{I_x} f\right|^p\,dx=\sum_{i=0}^{n-1}\int_{a_i}^{a_{i+1}}w(x)^p\left|\int_{I_{\nu,x}\cap [a_i,a_{i+1}]} f\right|^p\,dx\\
 &\le\sum_{i=0}^{n-1}\left[\int_{a_i}^{a_{i+1}}|h|\right]^p \int_{a_i}^{a_{i+1}}w^p=\frac{1}{n^p}\left[\int_a^b|h|\right]^p\int_a^b w^p<\varepsilon^p.
 \end{align*}
 \end{proof}

\begin{cor} Let \eqref{polozh} hold, $g\in\mathfrak{M}(I)$. If ${\cal L}^1(\{x\in I:g(x)\not=0\})>0$, then $\mathbf{J}_{\Csp}(g)=\infty$.
\end{cor}

\begin{proof}
 Fix $g\in\mathfrak{M}(I)$ with ${\cal L}^1(\{x\in I:g(x)\not=0\})>0$ and arbitrary $\varepsilon>0$. There exists  $[a,b]\subset I$ such that $a<b$ and ${\cal L}^1((a,b)\cap \{x\in I:g(x)\not=0\})>0$.  By Lemma~\ref{chasto} there exists $f\in \Csp$ with the properties: $\|f\|_{\Csp}<\varepsilon$ and $|f|=1$ on $(a,b)$. Hence 
 $$\frac{\int_I|fg|}{\|f\|_{\Csp}}\ge \frac{1}{\varepsilon}\int_a^b |g|.$$
 Thus, $\mathbf{J}_{\Csp}(g)=\infty$.
 \end{proof}

 \begin{lem}\label{neprg} Let \eqref{polozh} hold, $g\in\mathfrak{M}(I)$ is  continuous on $I$ and  $J_{\Ces}(g)<\infty$. If $\nu=1$ then $g(d-0)=0$, if $\nu=2$ then $g(c+0)=0$. 
\end{lem}

 \begin{proof}
The case $\nu=2$. Assume that $\delta_0:=\limsup_{s\to c+0}|g(s)|>0$. Let $\delta\in (0,\delta_0)$ and $E_\delta:=\{x\in I: |g(x)|>\delta\}$.
Denote $N:=\sup\{k\in\Ze: 2^k\le\int_I w^p\}+1$ and define $\{a_k\}_{k<N}$ by equality $\int_{I_{\nu,a_k}^*} w^p=2^k$. Let  $k_0<\min\{N,0\}$ and $A:=\{j\le k_0:[a_{j-1},a_j]\cap E_\delta\not=\emptyset\}$. Then the set $A$ is infinite and  for any $j\in A$ there exists $[\alpha_j,\beta_j]\subset [a_{j-1},a_j]\cap E_\delta$ such that $\alpha_j<\beta_j$ and $\sgn g=const$ on $[\alpha_j,\beta_j]$. We take $\gamma\in(1,2^\frac{1}{p})$ and put 
$$f(x):=\sum_{j\in A}\frac{\gamma^{-j}}{\beta_j-\alpha_j}\sgn g(x)\chi_{[\alpha_j,\beta_j]}(x),\ x\in I.$$
Then
\begin{align*}
&\|f\|_{\Ces}^p\le \sum_{k\le k_0}\int_{a_{k-1}}^{a_k}w^p\left[\int_{a_{k-1}}^{a_{k_0}}|f|\right]^p\le \sum_{k\le k_0}2^k\left[\sum_{j=k}^{k_0}\int_{a_{j-1}}^{a_j}|f|\right]^p\\
&\le \sum_{k\le k_0}2^k\left[\sum_{j=k}^0\gamma^{-j}\right]^p
\le (\gamma-1)^{-p}\sum_{k\le 0}2^k [\gamma^{-k+1}-1]^p<\infty
\end{align*}
and
$$\int_I gf=\int_I|gf|\ge \delta \sum_{j\in A}\gamma^{-j}=\infty.$$
Thus we get  contradiction with $J_{\Ces}(g)<\infty$.
 \end{proof}

 \section{Characterization of dual spaces for $\Csp$}

 \begin{lem}\label{neobkh}
  Let \eqref{polozh} hold,  $g\in \mathfrak{M}(I)$ and $J_{\Csp}(g;C_0^\infty(I))<\infty$.
  Then  $g$ admits  a representative $\tilde g\in AC_\loc(I)$,  the function $D\tilde g$ belongs $L^{p'}_\frac{1}{w}(I)$, and $J_{\Csp}(g;C_0^\infty(I))\ge \|D\tilde g\|_{L^{p'}_\frac{1}{w}(I)}$.
 \end{lem}

 \begin{proof}  Fix an arbitrary $\phi\in C_0^\infty(I)$. Put 
 $f:=D\phi$.  Then $f\in C_0^\infty(I)$ and $\|f\|_{\Csp}=\|w\phi\|_{L^p(I)}<\infty$.  Hence 
 \begin{align*}
M:=\sup_{\phi\in C_0^\infty(I)}\frac{|\int_I g D\phi|}{\|w\phi\|_{L^p(I)}}\le J_{\Csp}(g;C_0^\infty(I))<\infty.  
 \end{align*}
 Put $\Lambda \phi:=\int_I g D\phi$, $\phi\in C_0^\infty(I)$. By the Hahn\,--\,\,Banach theorem there exists an extension $\tilde \Lambda\in (L^p_w(I))^*$ of the functional $\Lambda$ with $\|\Lambda\|_{(L^p_w(I))^*}\le M$. By the Riesz theorem there exists $v\in L^{p'}_\frac{1}{w}(I)$ such that $\tilde \Lambda h=-\int_I vh$, $h\in L^p_w(I)$. Hence $v\in L^1_\loc(I)$ and
 $$-\int_I v\phi=\int_I g D\phi,\ \ \phi\in C_0^\infty(I).$$
 Therefore $v$ is a weak derivative of the function $g$. In particular, (see \cite[Theorem 7.13]{Leoni}), the function $g$ admits  a representative $\tilde g\in AC_\loc(I)$,
 $D\tilde g\in L^{p'}_\frac{1}{w}(I)$
 and 
 $$J_{\Csp}(g;C_0^\infty(I))\ge \sup_{\phi\in C_0^\infty(I)}\frac{|\int_I \phi D\tilde g|}{\|w\phi\|_{L^p(I)}}=\|D\tilde g\|_{L^{p'}_\frac{1}{w}(I)}.$$
\end{proof}

\begin{theo}\label{cesp}
Let \eqref{polozh} hold, $g\in \mathfrak{M}(I)$. Then $J_{\Csp}(g;\Ces)<\infty$ if and only if 
$g$ admits  a representative $\tilde g\in AC_\loc(I)$ of the form $\tilde g(x)=(-1)^\nu\int_{I_{\nu,x}^*} D\tilde g$, $x\in I$, and $D\tilde g\in L^{p'}_\frac{1}{w}(I)$. In this case, 
$J_{\Csp}(g;\Ces)=\|D\tilde g\|_{L^{p'}_\frac{1}{w}(I)}$.
\end{theo}

\begin{proof} {\it Necessity.}  By Lemma~\ref{neobkh}  $g$ admits  a representative $\tilde g\in AC_\loc(I)$, and   $D\tilde g\in L^{p'}_\frac{1}{w}(I)$. By  H\"older's inequality $D\tilde g\in   L^1_{\loc,(\nu,2)}$.
 
Further, 
 $$J_{\Ces}(\tilde g)\le  J_{\Csp}(g;\Ces)<\infty.$$
 By Lemma \ref{neprg} we have  $\tilde g(d-0)=0$ if $\nu=1$, or $\tilde g(c+0)=0$ if $\nu=2$. 
 From $\tilde g\in AC_\loc(I)$ we obtain the representation
 $\tilde g(x)=\tilde g(s)-\int_x^s D\tilde g$ for any  $s\in I$. Passing to the limit as  $s\to d-0$ in case $\nu=1$, or   $s\to c+0$ in case $\nu=2$  we get $\tilde g(x)=(-1)^\nu\int_{I_{\nu,x}^*} D\tilde g$, $x\in I$.
 
 Also, by Lemma~\ref{neobkh} we have 
 $$J_{\Csp}(g;\Ces)\ge J_{\Csp}(g;C_0^\infty(I))\ge \|D\tilde g\|_{L^{p'}_\frac{1}{w}(I)}.$$
 
 {\it Sufficiency.} Fix an  arbitrary $f\in \Ces$.  We have
 \begin{align*}
\int_I |fg|&\le \int_I|f(x)|\left[\int_{I_{\nu,x}^*} |D\tilde g|\right]\,dx=\int_I |D\tilde g(x)|\int_{I_{\nu,x}}|f|\,dx\\
&\le \|D\tilde g\|_{L^{p'}_\frac{1}{w}(I)}\|f\|_{\Ces}<\infty.  
 \end{align*}
 Let   $\{\alpha_k\}_1^\infty$, $\{\beta_k\}_1^\infty$ be the sequences existence of which is proved in Lemma \ref{sequence}.  For $n\in\mathbb{N}$ we put $f_n:=f\chi_{[\alpha_n,\,\beta_n]}$. Then $f_n\in \Ces$ and  (see Corollary \ref{dences})
$\lim_{n\to\infty}\|f-f_n\|_{\Csp}=0$.

 By integrating by parts  we obtain 
$$\int_I f_ng=\int_I f_n\tilde g =-\int_I \left[\int_{I_{\nu,x}}f_n\right] D\tilde g(x)\,dx.$$ 
By the dominated convergence  theorem, we have
$\lim_{n\to\infty}\int_I  f_n\tilde g=\int_I f\tilde g$
and 
$$\left|\int_I fg\right|\le \limsup_{n\to\infty}\int_I \left|\int_{I_{\nu,x}}f_n\right| |D\tilde g(x)|\,dx\le\|D\tilde g\|_{L^{p'}_\frac{1}{w}(I)}\|f\|_{\Csp}.$$
Therefore, $J_{\Csp}(g;\Ces)\le \|D\tilde g\|_{L^{p'}_\frac{1}{w}(I)}<\infty$.
  \end{proof}

\noindent{\bf Remark.} The conditions of Theorem~\ref{cesp} are not sufficient for  finiteness $J_{\Csp}(g)$ due to the following reasons. If
 $\int_I\left|f(x)\int_{I_{\nu,x}^*} h\right|\,dx<\infty$ for any $h\in \mathfrak{M}(I)$ with $\|h\|_{L^{p'}_\frac{1}{w}(I)}<\infty$ then by \cite[Lemma 2.4]{PSU0} 
 $$\sup_{h\in \mathfrak{M}(I)}\frac{\int_I\left|f(x)\int_{I_{\nu,x}^*} h\right|\,dx}{\|h\|_{L^{p'}_\frac{1}{w}(I)}}<\infty.$$
However, 
  $$\sup_{h\in \mathfrak{M}(I)}\frac{\int_I\left|f(x)\int_{I_{\nu,x}^*} h\right|\,dx}{\|h\,\frac{1}{w}\|_{L^{p'}(I)}}=\sup_{h\in \mathfrak{M}(I),\,h\ge 0}\frac{\int_I |f(x)|\int_{I_{\nu,x}^*} h\,dx}{\|h\,\frac{1}{w}\|_{L^{p'}(I)}}=\|f\|_{\Ces}.$$
  
\begin{theo}\label{cesshtrih} Let \eqref{polozh} hold,  $g\in \mathfrak{M}(I)$. Then $J_{\Csp}(g)<\infty$ if and only if 
$g$ admits  a representative $\tilde g\in AC_\loc(I)\cap L^\infty(I)$ of the form $\tilde g(x)=(-1)^\nu\chi_{I_{\nu,b}}(x)\int_{I_{\nu,x}^*} D\tilde g$, $x\in I$, where $b\in I$, and $D\tilde g\in L^{p'}_\frac{1}{w}(I)$. In this case, 
$J_{\Csp}(g)=\|D\tilde g\|_{L^{p'}_\frac{1}{w}(I)}$.
  \end{theo}
  
 \begin{proof} {\it Sufficiency.} 
 Fix an  arbitrary $f\in \Csp$.  We have
 $$\int_I |fg|\le \|\tilde g\|_{L^\infty(I)}\int_{I_{\nu,b}}|f|<\infty.$$
 Let   $\{\alpha_k\}_1^\infty$, $\{\beta_k\}_1^\infty$ be the sequences existence of which is proved in Lemma \ref{sequence}.  For $n\in\mathbb{N}$ we put $f_n:=f\chi_{[\alpha_n,\,\beta_n]}$. Then $f_n\in \Ces$ and  (see Corollary \ref{dences})
$\lim_{n\to\infty}\|f-f_n\|_{\Csp}=0$. By the dominated convergence  theorem, we have
$\lim_{n\to\infty}\int_I  f_n\tilde g=\int_I f\tilde g$
and  by Theorem~\ref{cesp} we obtain
$$\left|\int_I fg\right|\le \lim_{n\to\infty}\left|\int_I f_n\tilde g\right|\le\|D\tilde g\|_{L^{p'}_\frac{1}{w}(I)} \|f\|_{\Csp}.$$
Therefore, $J_{\Csp}(g)\le \|D\tilde g\|_{L^{p'}_\frac{1}{w}(I)}<\infty$.

 {\it Necessity.} Denote $E:=\{x\in I:|\tilde g(x)|>0\}$. Since $|\tilde g|$ is continuous on $I$, the set $E$ is open.  Assume that ${\cal L}^1(I_{\nu,b}^*\cap E)>0$ for any $b\in I$.  Then there exists sequence $\{[a_k,b_k]\}_1^\infty$ of disjoint intervals such that $m_k:=\min_{x\in [a_k,b_k]}|\tilde g(x)|>0$. Put $\theta_k:=\frac{b_k-a_k}{k m_k}$. 
 By Lemma~\ref{chasto} there exists $f_k\in \Ces$ with properties: $\|f_k\|_{\Csp}<2^{-k}$, $\supp f_k\subset[a_k,b_k]$ and $|f_k|=\theta_k$ on $(a_k,b_k)$.
 Then for the function
 $f:=\sum_{k=1}^\infty f_k$ we have $\|f\|_{\Csp}\le 1$
 and
 $$\int_I|fg|\ge \sum_{k=1}^\infty \theta_k m_k(b_k-a_k)=\sum_{k=1}^\infty\frac{1}{k}=\infty.$$
 This contradicts the relation $J_{\Csp}(g)<\infty$.

 Fix $a\in I$. Assume that  $g\not\in L^\infty(I_{\nu,a})$. Then there exists  $h\in L^1(I_{\nu,a})$ such that $\int_{I_{\nu,a}} |hg|=\infty$. Let $\{a_k\}_{k=0}^\infty$, where $a_0=a$, be partition of $I_{\nu,a}$.
 By Lemma~\ref{chasto} there exists $f_k\in \Ces$ with properties: $\|f_k\|_{\Csp}<2^{-k}$, $\supp f_k\subset(I_{\nu,a_k}\setminus I_{\nu,a_{k+1}})$  and $|f_k|=|h|$ on interior of $I_{\nu,a_k}\setminus I_{\nu,a_{k+1}}$.
 Then for the function 
 $f:=\sum_{k=1}^\infty f_k$ we have $\|f\|_{\Csp}\le 1$ and
$$\int_I|fg|\ge \int_{I_{\nu,a}}|hg|=\infty.$$
This contradicts the relation $J_{\Csp}(g)<\infty$.
  
  The rest of the statements follow from Theorem~\ref{cesp}.
  \end{proof}

\begin{cor}
 Let \eqref{polozh} hold and $\nu=1$. Then  for each $f\in L^1_{\loc,(\nu,1)}$ the equalities $\|f\|_{(\Csp)^\downarrow_{w,w}}=\|f\|_{\Csp}$ and  
 $\|f\|_{(\Csp)^\downarrow_{w,s}}=\|f\|_{\Ces}$ hold. Moreover, $(\Csp)^\downarrow_{w,w}=\Csp$, $(\Csp)^\downarrow_{w,s}=\Ces$.
 \end{cor}

\begin{proof} Fix an arbitrary $f\in L^1_{\loc,(\nu,1)}$. Let $g\in \mathfrak{M}(I)$  be a non-negative and non-increasing function. By Theorem \ref{cesshtrih} $g\in (\Csp)'_w$ if and only if  
$g$ admits  a representative $\tilde g\in AC_\loc(I)\cap L^\infty(I)$ of the form $\tilde g(x)=(-1)^\nu\chi_{I_{\nu,b}}(x)\int_{I_{\nu,x}^*} D\tilde g$, $x\in I$, where $b\in I$,  $D\tilde g$ is non-negative and belongs to $L^{p'}_\frac{1}{w}(I)$. Hence
\begin{align*}
\|f\|_{(\Csp)^\downarrow_{w,w}}&=\sup_{b\in I}\sup_{0\le h\in L^{p'}_\frac{1}{w}(I_{\nu,b})\cap L^1(I_{\nu,b})}\frac{\left|\int_{I_{\nu,b}} f(x)\left[\int_{I_{\nu,b}\cap I_{\nu,x}^*} h\right]\,dx\right|}{\|h\|_{L^{p'}_\frac{1}{w}(I_{\nu,b})}}\\
&=\|f\|_{\Csp} 
\end{align*}
and, analogously, $\|f\|_{(\Csp)^\downarrow_{w,s}}=\|f\|_{\Ces}$.

Since for any $b\in I$ there is $h\in (\Csp)'_w$ such that $h(x)=1$ for $x\in I_{\nu,b}$, then    $(\Csp)^\downarrow_{w,w}\subset L^1_{\loc,(\nu,1)}$ and $(\Csp)^\downarrow_{w,s}\subset L^1_{\loc,(\nu,1)}$. Thus, $(\Csp)^\downarrow_{w,w}=\Csp$ and $(\Csp)^\downarrow_{w,s}=\Ces$.
\end{proof}

\begin{theo} Let \eqref{polozh} hold. For any $\Lambda\in (\Csp)^*$ there exists $h\in L^{p'}_\frac{1}{w}(I)$, such that 
\begin{equation}\label{predstfunk}
\Lambda f=\int_I h(x)\left[\int_{I_{\nu,x}}f\right]\,dx,\ \ f\in\Csp, 
\end{equation}
and $\|\Lambda\|_{(\Csp)^*}=\|h\|_{L^{p'}_\frac{1}{w}(I)}$.
  \end{theo}

\begin{proof}   Fix an arbitrary  $\Lambda\in (\Csp)^*$.  Let $\bar\Lambda$ be restriction $\Lambda$ on $\Ces$. Then $\bar\Lambda\in (\Ces)^*$ and there exists function $g\in (\Ces)'_s$ such that $\bar\Lambda f=\int_I fg$ for any $f\in\Ces$. Consequently, for any $f\in\Ces$ we have
 $$\left|\int_I fg\right|=|\bar\Lambda f|=|\Lambda f|\le \|\Lambda\|_{(\Csp)^*}\|f\|_{\Csp},$$  
 that is $J_{\Csp}(g;\Ces)\le \|\Lambda\|_{(\Csp)^*}<\infty$.
 By Theorem \ref{cesp} $g$ admits  a representative $\tilde g\in AC_\loc(I)$ of the form $\tilde g(x)=(-1)^\nu\int_{I_{\nu,x}^*} D\tilde g$, $x\in I$,  $D\tilde g\in L^{p'}_\frac{1}{w}(I)$, and $J_{\Csp}(g;\Ces)=\|D\tilde g\|_{L^{p'}_\frac{1}{w}(I)}$.

 Fix any $f\in\Csp$. Let   $\{\alpha_k\}_1^\infty$, $\{\beta_k\}_1^\infty$ be the sequences existence of which is proved in Lemma \ref{sequence}.  For $n\in\mathbb{N}$ we put $f_n:=f\chi_{[\alpha_n,\,\beta_n]}$. Then $f_n\in \Ces$,  
$\lim_{n\to\infty}\|f-f_n\|_{\Csp}=0$ (see Corollary \ref{dences}) and, consequently, $\Lambda f=\lim_{n\to\infty}\Lambda f_n$.

The case $\nu=2$. We write
\begin{align*}
&\Lambda f_n=\bar\Lambda f_n=(-1)^\nu\int_{\alpha_n}^{\beta_n}f(x)\left[\int_{I_{\nu,x}^*} D\tilde g\right]\,dx\\
&=\int_I D\tilde g (y)\left[\int_{\alpha_n}^{\beta_n} f(x)\chi_{(c,x]}(y)\,dx\right]\,dy=\int_c^{\beta_n} D\tilde g (y)\left[\int_{\max\{\alpha_n,y\}}^{\beta_n} f\right]\,dy\\
&=\left[\int_c^{\alpha_n} D\tilde g\right] \left[\int_{\alpha_n}^d f-\int_{\beta_n}^d f\right]+\int_{\alpha_n}^{\beta_n} D\tilde g (y)\left[\int_y^d f-\int_{\beta_n}^d f\right]\,dy\\
&=\int_{\alpha_n}^{\beta_n}D\tilde g(y)\left[\int_{I_{\nu,y}} f\right]\,dy+\left[\int_{I_{\nu,\alpha_n}} f\right]\int_{I_{\nu,\alpha_n}^*} D\tilde g-\left[\int_{I_{\nu,\beta_n}} f\right]\int_{I_{\nu,\beta_n}^*} D\tilde g.
\end{align*}
Analogously in case $\nu=1$ we have
\begin{align*}
&\Lambda f_n=\bar\Lambda f_n=(-1)^\nu\int_{\alpha_n}^{\beta_n}f(x)\left[\int_{I_{\nu,x}^*} D\tilde g\right]\,dx\\
&=(-1)^\nu\int_{\alpha_n}^{\beta_n}D\tilde g(y)\left[\int_{I_{\nu,y}} f\right]\,dy-(-1)^\nu\left[\int_{I_{\nu,\alpha_n}} f\right]\int_{I_{\nu,\alpha_n}^*} D\tilde g+\\
&+(-1)^\nu\left[\int_{I_{\nu,\beta_n}} f\right]\int_{I_{\nu,\beta_n}^*} D\tilde g.
\end{align*}
Since $D\tilde g\in L^{p'}_\frac{1}{w}(I)$ by Lemma \ref{sequence} we have
$\lim_{n\to\infty}\left|\int_{I_{\nu,\alpha_n}}f \cdot \int_{I_{\nu,\alpha_n}^*} D\tilde g\right|=0$ and $\lim_{n\to\infty}\left|\int_{I_{\nu,\beta_n}}f \cdot \int_{I_{\nu,\beta_n}^*} D\tilde g\right|=0$. Besides that $(D\tilde g(\cdot) \int_{I_{\nu,\cdot}}f)\in L^1(I)$ and by Lebesgue dominated theorem we get 
$$\lim_{n\to\infty}\int_{\alpha_n}^{\beta_n}D\tilde g(x) \left[\int_{I_{\nu,x}}f\right]\,dx=\int_I D\tilde g(x) \left[\int_{I_{\nu,x}}f\right]\,dx.$$
Putting $h:=(-1)^\nu D\tilde g$, we have \eqref{predstfunk} and $\|\Lambda\|_{(\Csp)^*}=\|h\|_{L^{p'}_\frac{1}{w}(I)}$. 
\end{proof}

\end{document}